\documentclass[a4paper,10pt]{amsart}

\usepackage{mathrsfs}
\usepackage{amsrefs}
\usepackage{amssymb}
\usepackage{amsmath}
\usepackage{amsthm}
\usepackage{color}
\usepackage{enumerate}
\input arrow

\DeclareMathOperator{\im}{Im} 

\DeclareMathOperator{\Hom}{Hom}
\DeclareMathOperator{\End}{End}
\DeclareMathOperator{\add}{add}
\DeclareMathOperator{\Add}{Add}

\DeclareMathOperator{\re}{r}
\DeclareMathOperator{\ex}{e}
\DeclareMathOperator{\Soc}{Soc}

\newcommand{\Modr}{\textrm{Mod-}R}

\newcommand{\homc}{\Hom_{\mathbf C}}

\newtheorem{definition}{Definition}[section]
\newtheorem{proposition}[definition]{Proposition}
\newtheorem{theorem}[definition]{Theorem}
\newtheorem{example}[definition]{Example}

\newtheorem{lemma}[definition]{Lemma}
\newtheorem{corollary}[definition]{Corollary}
\newtheorem{remark}[definition]{Remark}

\title{MAXIMAL IDEALS IN MODULE CATEGORIES AND APPLICATIONS}

 \author{Manuel Cort\'es-Izurdiaga}
 \address{Department of Mathematics, University of Almeria, E-04071,
   Almeria, Spain}
\thanks{The first author is partially supported by projects
  MTM2014-54439 and MTM2016-77445-P from MEC
  and by research group FQM211 from Junta de Andaluc\'{\i}a.}
\email{mizurdia@ual.es}
 \author{Alberto Facchini}
    \address{Dipartimento di Matematica, Universit\`a di Padova, 35121 Padova, Italy}
   \email{facchini@math.unipd.it}
\thanks{The second author is partially
supported by Dipartimento di Matematica, Universit\`a di Padova (Progetto SID 2016 BIRD163492/16 ``Categorical homological methods in the study of algebraic structures" and progetto DOR1690814 ``Anelli e categorie di moduli'').}

\begin{document}

\begin{abstract} We study the existence of maximal ideals in
  preadditive categories defining an order $\preceq$ between
  objects, in such a way that if there do not exist
  maximal objects with respect to $\preceq$, then there is no
  maximal ideal in the category. In our study, it is sometimes sufficient to restrict our
  attention to suitable subcategories. We give an example of
  a category $\mathbf C_F$ of modules over a right noetherian ring
  $R$ in which there is a unique maximal ideal. The category $\mathbf C_F$ is related to an indecomposable
  injective module $F$, and the objects of
  $\mathbf C_F$ are the $R$-modules of finite $F$-rank.
\end{abstract}
\maketitle

\section*{INTRODUCTION}
\label{sec:introduction}

This paper is related to the study of ideals in preadditive
categories. Recall that an {\em ideal} in a preadditive category
$\mathbf C$ is an additive subfunctor $\mathcal I$ of the additive
bifunctor
$\homc\colon \mathbf C^{op}\times \mathbf C \rightarrow \mathbf{Ab}$,
where $\mathbf{Ab}$ is the category of abelian groups.

Let us mention two motivations for our study. The first is related to
extensions of the classical Krull-Schmidt theorem to additive
categories. In \cite{Facchini}, the second author proved that the
class of all uniserial right modules over a ring $R$ does not satisfy
the Krull-Schmidt theorem, thus answering a question posed by Warfield in
1975, but that nevertheless a weak version of the Krull-Schmidt theorem for
uniserial modules holds \cite[Theorem 1.9]{Facchini}.  This weak
version of the Krull-Schmidt theorem was extended as follows, in
\cite[Theorem 6.4]{FacchiniPerone2}, to any additive category
$\mathbf A$ with a pair of ideals $\mathcal I$ and $\mathcal J$
satisfying suitable conditions: if $U_1, \ldots, U_n$ and
$V_1, \ldots, V_n$ are objects in $\mathbf A$ with local endomorphism
rings in the quotient categories $\mathbf A/\mathcal I$ and
$\mathbf A/\mathcal J$, then
$U_1 \oplus \cdots \oplus U_n \cong V_1 \oplus \cdots V_m$ if and only
if $n=m$ and there exist two permutations $\sigma$ and $\tau$ of
$\{1, \ldots, n\}$ such that $U_i$ and $V_{\sigma(i)}$ are isomorphic
in $\mathbf A/\mathcal I$, and $U_i$ and $V_{\tau(i)}$ are isomorphic
in $\mathbf A/\mathcal J$, for every $i =1, \ldots, n$.

Our second motivation is related to the problem of approximating
objects by morphisms belonging to some ideal. This idea first appeared
in \cite{Herzog}, where the author introduced phantom maps in module
categories, considered the ideal consisting of all such maps and
proved that each module $M$ has a phantom cover (that is, a phantom
map $\varphi\colon P \rightarrow M$ such that every phantom map
$\psi\colon Q \rightarrow M$ factors through $\varphi$, and minimal
with respect to this property). This particular situation was extended
in \cite{FuGuilHerzogTorrecillas}, where it was characterized when an
ideal $\mathcal I$ in an exact category provides approximations in
this sense. Notice that this theory contains, as a particular case,
the classical one about precovers and covers by objects, see
\cite{EnochsJenda}.

As in the case of ideals of rings, one can consider minimal and
maximal ideals in a preadditive category $\mathbf C$. In \cite[Theorem
3.1]{Facchini2}, it is proved that the minimal ideals in a module
category are in one-to-one correspondence with the simple
modules. Hence we have a complete description of the minimal ideals of
the category. A similar description of maximal ideals is not
known (the best description of maximal ideals is Prihoda's result
\cite[Lemma~2.1]{FacchiniPerone}). One of the main results of our
paper is now that there do not exist maximal ideals in module
categories Mod-$R$ (actually, in Grothendieck categories). The idea of the
proof is to define a order $\preceq$ in the class of objects
and relate the existence of maximal ideals with the existence of
non-zero maximal objects with respect to this order. More precisely, we prove (Theorem
\ref{t:MainTheorem}) that if for each object $A$ in the category there
exists an object $B$ such that $A \prec B$, then there do not
exist maximal ideals. Since a Grothendieck category has this property (Proposition \ref{p:ExistenceBigObjectsGrothendieck}), we
conclude that there are no maximal ideals in Grothendieck
categories.

If $\mathbf C$ is a preadditive category, we can
consider the full subcategory $\mathbf M(\mathbf C)$ of $\mathbf C$
consisting of all objects $C$ of $\mathbf C$ for which there do not
exist objects $B$ in $\mathbf C$ with $C\prec B$. Then the maximal ideals of
$\mathbf M(\mathbf C)$ determine those of $\mathbf C$ (Proposition
\ref{p:MaximalMC}). Using these ideas, the last part of the paper is
devoted to describing the maximal ideals in a full subcategory
$\mathbf C_F$ constructed starting from an indecomposable injective
module $F$ over a right noetherian ring.

All rings in this paper are associative with unit and not necessarily
commutative. If $R$ is such a ring, module will mean right $R$-module
and we will denote by $\Modr$ the category whose objects are all right
$R$-modules.

\section{PRELIMINARIES}
\label{sec:preliminaries}

By a {\em preadditive category}, we mean a category together with an
abelian group structure on each of its hom-sets such that composition
is bilinear. An {\em additive category} is a preadditive category with
finite products. Let $\mathbf C$ be a preadditive category and $A$ an
object of $\mathbf C$.  We will denote by $\add(A)$ the class of the
objects $X$ of $\mathbf C$ for which there exist an integer $n>0$ and
morphisms $f_1, \ldots, f_n \in \Hom_\mathbf{C}(A,X)$ and
$g_1, \ldots, g_n \in \Hom_\mathbf{C}(X,A)$ such that
$1_X = \sum_{i=1}^nf_ig_i$. If $\mathbf C$ is additive and idempotents
split in $\mathbf C$, then $X \in \add(A)$ if and only if $X$ is
isomorphic to a direct summand of $A^n$ for some integer $n\ge0$. If,
moreover, $\mathbf C$ has arbitrary direct sums, we will denote by
$\Add(A)$ the class of all objects that are isomorphic to direct
summmands of arbitrary direct sums of copies of $A$.

An {\em ideal} in $\mathbf C$ is an additive subfunctor $\mathcal I$
of the additive bifunctor
$\homc\colon \mathbf C^{op}\times \mathbf C \rightarrow \mathbf{Ab}$,
where $\mathbf{Ab}$ is the category of abelian groups. Thus
$\mathcal I$ associates to every pair $A$ and $B$ of objects in
$\mathbf C$ a subgroup $\mathcal I(A,B)$ of $\homc(A,B)$ so that if
$f\colon X \rightarrow A$ and $g\colon B \rightarrow Y$ are morphisms
in $\mathbf C$ and $i \in \mathcal I(A,B)$, then
$gif \in \mathcal I(X,Y)$. An ideal in $\mathbf C$ is {\em maximal} if
it is proper, that is, it is not equal to $\homc$, and is not properly
contained in any other proper ideal. For instance, it is easy to see
that the zero ideal is a maximal ideal in the full subcategory of
$\textrm{Mod-}K$ whose objects are all finite-dimensional vector
spaces over a field $K$.

Given an object $A$ in $\mathbf C$ and any two-sided ideal $I$ of
$\End_{\mathbf C}(A)$, we will denote by $\mathcal A_I$ the ideal of
the category $\mathbf C$ defined, for each pair of objects
$X,Y \in \mathbf C$, by
\[\begin{array}{l}
    \mathcal A_I(X,Y) = \{f \in \Hom_{\mathbf C}(X,Y):\beta f \alpha \in
    I \textrm{ for}\\ \qquad\qquad\qquad\textrm{ all } \alpha \in \Hom_{\mathbf C}(A,X) \textrm{ and }\beta \in
    \Hom_{\mathbf C}(Y,A)\}.\end{array}\]
  This ideal is called the {\em  ideal associated to $I$} (\cite[Section~2]{FacchiniPrihoda3} and \cite[Section~3]{FacchiniPrihoda4}). The ideal
  $\mathcal A_I$ contains any ideal $\mathcal I$ in $\mathbf C$
  satisfying $\mathcal I(A,A) \subseteq I$. As proved in
  \cite[Lemma 2.4]{FacchiniPerone}, there is a strong relation
  between ideals associated to maximal ideals of the endomorphism ring
  of an object, and maximal ideals in the preadditive category. For
  instance, the same argument as \cite[Proposition 2.5]{FacchiniPerone} gives:

  \begin{example}\label{e:MaximalIdeals}{\rm Let $\mathbf C$ be an
      additive category in which idempotent splits and $C$ any object
      of $\mathbf C$. Then the maximal ideals in the category
      $\add(C)$ are the ideals associated to maximal ideals of
      $\End_{\mathbf C}(C)$.}
  \end{example}

  The following easy lemma will be useful to compute ideals in the
  endomorphism ring of a finite direct sum of objects.

\begin{lemma}\label{l:IdealsInDirectSums}
  Let $\mathbf C$ be an additive category, $A$ an object of
  $\mathbf C$ and $I$ an ideal in $\End_{\mathbf C}(A)$. Given any
  finite family $B_1, \ldots, B_n$ of objects of $\mathbf C$, denote
  by $\iota_l$ and $\pi_l$ the inclusion and the projection
  corresponding to the $l$-th component of $B=\bigoplus_{i=1}^n B_i$
  for each $l =1,\dots,n$. Then
  \begin{displaymath}
    \mathcal A_I(B,B) = \{\,f \in \End_{\mathbf C}(B): \pi_m f
    \iota_l \in \mathcal A_I(B_l,B_m)  \ \textrm{for every } l,m =1,2,\dots,n\,\}.
  \end{displaymath}
\end{lemma}

Note that, as a consequence of this result, if $M_1$ and $M_2$ are
objects in an additive category $\mathbf C$ and $I$ is an ideal in the
endomorphism ring of an object $A$ of $\mathbf C$, then
$\mathcal A_I(M_1 \oplus M_2,M_1 \oplus M_2)= \End_R(M_1 \oplus M_2)$
if and only if $\mathcal A_I(M_i,M_i)=\End_R(M_i)$ for
$i =1,2$.

\section{The strict order $\prec$ and its corresponding partial order $\preceq$.}
\label{sec:maxim-ideals-module+}

The existence of maximal ideals in preadditive categories is related
to an order~$\preceq$ between objects. In this section, we define the partial order $\preceq$ and give a number of examples.

\begin{definition} \label{d:big} Let $\mathbf C$ be a preadditive
  category and $A$, $B$ objects of $\mathbf C$. Set $A\prec B$
  if there exists an infinite subset
  $E \subseteq \Hom_{\mathbf C}(B,A) \times \Hom_{\mathbf C}(A,B)$
  with the following properties:
  \begin{enumerate}
  \item $f g = 1_A$ for every $(f,g) \in E$.

  \item For each $\varphi \in \Hom_{\mathbf C}(A,B)$,
    $|\{(f,g) \in E:f\varphi \neq 0\}| < |E|$.
  \end{enumerate} We shall write $A\preceq B$ if either $A\prec B$ or $A=B$.
\end{definition}

Here we are using the well known one-to-one correspondence between strict orders and partial orders. For any partial order $\le$, the corresponding strict order $<$ is defined by $A<B$ if $A\le B$ and $A\ne B$.

\medskip

Let $\mathbf C$ be a preadditive category, $\mathbf A$ a subcategory
of $\mathbf C$ and $A$ and $B$ objects of $\mathbf A$. Notice that it
can occur that $A\prec B$ in $\mathbf C$ but not in
$\mathbf A$. However, if $\mathbf A$ is full, $A\prec B$ in $\mathbf C$ if and only if $A\prec B$ in
$\mathbf A$.

\begin{example}{\rm Let $\mathbf C$ be any preadditive category and
    $A,B\in \mathbf C$ objects. If both $\Hom_{\mathbf C}(A,B)$ and
    $\Hom_{\mathbf C}(B,A)$ are finite, then $A\not\prec B$. In particular, if $\mathbf C$ has a zero
    object $0$, then $0\not\prec B$ and $B \not\prec 0$ for every object $B$.}
\end{example}

Let us see some properties of the order $\preceq$. 

\begin{lemma}\label{l:PropertiesOfBigObjects}
  Let $\mathbf C$ be a preadditive category
  and $A$, $B$ and $C$ objects of $\mathbf C$.
  \begin{enumerate}
  \item If $A\prec B$ and $B$ is a retract of $C$,
    then $A\prec C$.

  \item If $B \in \add(A)$, then $A\not\prec B$.
  \end{enumerate}
\end{lemma}

\begin{proof}
  {\em (1)} Denote by $\iota_B\colon B \to C$ and
  $\pi_B\colon C \to B$ the morphisms satisfying
  $\pi_B\iota_B=1_B$. Since $A\prec B$, there exists
  a set
  $E \subseteq \Hom_{\mathbf C}(B,A) \times \Hom_{\mathbf C}(A,B)$
  satisfying the conditions of Definition \ref{d:big}. Then
  $E'=\{(f \pi_B,\iota_B g):(f,g) \in E\}$ is a subset of
  $\Hom_{\mathbf C}(B\oplus C,A) \times \Hom_{\mathbf C}(A,B\oplus C)$
  that has cardinality equal to $|E|$ and that trivially verifies the
  conditions of Definition \ref{d:big}. Thus $A\prec C$.
  
  {\em (2)} Let $n>0$ be an integer and $$f_1, \ldots, f_n \in
  \Hom_{\mathbf C}(A,B),\quad g_1, \ldots, g_n \in \Hom_{\mathbf C}(B,A)$$ be 
  such that $\sum_{i=1}^nf_ig_i=1_B$. Suppose, in
  order to get a contradiction, that $A\prec B$. Let
  $E \subseteq \Hom_\mathbf{C}\left(B,A\right) \times
  \Hom_\mathbf{C}\left(A,B\right)$
  be the set satisfying the conditions of Definition
  \ref{d:big}. By Definition \ref{d:big}{\em (2)}, the set
  \begin{displaymath}
    E_k:=\{(f,g) \in E: ff_k \neq 0\}
  \end{displaymath}
  has cardinality smaller than $|E|$ for each $k =1, \ldots, n$. But, for each morphism
  $\varphi:B \rightarrow A$, $\varphi \neq 0$ if and only if
  $\varphi f_k \neq 0$ for some $k =1, \ldots,  n$. This implies that $E
  = \bigcup_{k=1}^nE_k$ as $f \neq 0$ for each $(f,g) \in
  E$. Since $E$ is infinite, we conclude that at least one
  of the sets $E_k$ has the same cardinality as $E$, which is a
  contradiction.
\end{proof}

Let $\mathbf C$ be a preadditive category. The main consequence of the
preceeding result is that the relation $\prec$ is a strict order, since it is irreflexive by {\em (2)}
and transitive by {\em (1)}. As we have already said, we denote by $\preceq$ the partial order
associated to the strict order $\prec$.

Now we will consider a relation between large direct sums of copies of a non-zero object in a
Grothendieck category and the strict order $\prec$ of Definition
\ref{d:big}. Let $\mathbf G$ be a Grothendieck category, $A$ an object
of $\mathbf G$ and $\kappa$ an infinite regular cardinal. Recall that
$A$ is said to be {\em $< \kappa$-generated} \cite[Definition
1.67]{AdamekRosicky} if $\Hom_{\mathbf G}(A, -)$ commutes with
$\kappa$-directed colimits with all morphisms in the direct system being
monomorphisms (a $\kappa$-directed colimit is the colimit of a
$\kappa$-system in $\mathbf G$, $(A_i,f_{ij})_{I}$, the latter meaning
that each subset of $I$ of cardinality smaller than $\kappa$ has an
upper bound \cite[Definition 1.13]{AdamekRosicky}).

\begin{proposition} \label{p:ExistenceBigObjectsGrothendieck} Let
  $\mathbf G$ be a Grothendieck category and $\kappa$ an infinite
  regular cardinal.
  \begin{enumerate}
  \item Let $A$ be a non-zero $< \kappa$-generated object of
    $\mathbf G$. Then $A \prec A^{(\kappa)}$.

  \item For each non-zero object $A$ of $\mathbf G$, there exists an
    object $B$ of $\mathbf G$ such that $A \prec B$.
  \end{enumerate}
\end{proposition}

\begin{proof}
  {\em (1)} Denote by $\iota_\alpha\colon A \rightarrow A^{(\kappa)}$
  and $\pi_\alpha\colon A^{(\alpha)} \rightarrow A$ the injection and
  the projection corresponding to the $\alpha$-component of
  $A^{(\kappa)}$ for each $\alpha < \kappa$. Consider the subset
  $\{(\pi_\alpha,\iota_\alpha):\alpha < \kappa\}$ of
  $\Hom_{\mathbf G}\left(A^{(\kappa)},A\right) \times \Hom_{\mathbf
    G}\left(A,A^{(\kappa)}\right)$.
  Then $E$ satisfies \textit{(1)} of Definition \ref{d:big} since
  $\pi_\alpha \iota_\alpha = 1_A$ for each $\alpha < \kappa$.

  In order to prove condition {\em (2)} of Definition \ref{d:big},
  note that $A^{(\kappa)}$ is the colimit of the $\kappa$-direct
  system $(A^{(\alpha)},\iota_{\alpha\beta})_{\kappa}$, where
  $\iota_{\alpha\beta}\colon A^{(\alpha)} \rightarrow A^{(\beta)}$ is
  the inclusion for each $\alpha < \beta$ in $\kappa$. The colimit maps
  are the inclusions
  $\iota_\alpha\colon A^{(\alpha)} \rightarrow A^{(\kappa)}$ for each
  $\alpha < \kappa$. Let $\varphi\colon A \rightarrow A^{(\kappa)}$ be
  any morphism. Since $A$ is $< \kappa$-generated and the morphism
  $\iota_{\alpha\beta}$ is monic for every $\alpha < \beta$ in
  $\kappa$, there exists $\alpha_0 < \kappa$ and
  $\overline \varphi\colon A \rightarrow A^{(\alpha_0)}$ such that
  $\varphi = \iota_{\alpha_0}\overline \varphi$. In particular, we get
  that
  \[|\{(\pi_\alpha,\iota_\alpha)\colon \pi_\alpha\varphi \neq 0\}|
  \leq |\alpha_0| < \kappa = |E|.\]

  {\em (2)} Notice that, for each object $A$ in $\mathbf G$, there
  exists an infinite regular cardinal $\kappa$ such that $A$ is
  $< \kappa$-generated \cite[Lemma A.1]{Stovicek}.
\end{proof}

Using these results, we can characterize when $V \prec W$ for vector
spaces $V$ and $W$.

\begin{corollary}\label{p:BigVectorSpaces}
  Let $V$ and $W$ be two vector spaces over a field $K$. Then $V \prec W$ if and only if $W$ is infinite dimensional and
  $0 \neq \dim(V) < \dim(W)$.
\end{corollary}

\begin{proof}
  Suppose $V \prec W$. First of all, note that
  $\textrm{dim}(V)<\textrm{dim}(W)$ since, otherwise, there would
  exist an epimorphism $\varphi\colon V \rightarrow W$. This would
  imply that, for any subset $E$ of $\Hom_K(W,V) \times \Hom_K(V,W)$,
  $\{(f,g) \in E:f\varphi \neq 0\} = E$. That is, $V\not\prec W$. Furthermore, $W$ has to be infinite dimensional,
  since finite dimensional vector spaces belong to $\add(V)$ and, by
  Lemma \ref{l:PropertiesOfBigObjects}, for any vector space $W$ in $\add(V)$, we have that $V\not\prec W$.

  Conversely, suppose that $W$ is infinite dimensional and that
  $0 \neq \dim(V) < \dim(W)$. Set $\textrm{dim }W = \kappa$ and
  $\textrm{dim }V = \lambda$ and take an infinite regular cardinal
  $\mu$ with $\lambda < \mu \leq \kappa$ (if $\kappa$ is regular, take
  $\mu = \kappa$; otherwise, set $\mu = \lambda^+$, the successor
  cardinal of $\lambda$). By Proposition
  \ref{p:ExistenceBigObjectsGrothendieck} and Lemma
  \ref{l:PropertiesOfBigObjects},
  $V \prec V^{(\mu)} \prec V^{(\mu)} \oplus V^{(\kappa)}$. Since
  $\lambda < \mu \leq \kappa$,
  $\dim\left(V^{(\mu)} \oplus V^{(\kappa)}\right)  = \kappa$ and $V^{(\mu)} \oplus V^{(\kappa)} \cong W$. Consequently, $V \prec W$.
\end{proof}

Let $R$ be a ring and $A$ and $B$ right $R$-modules. If $A \prec B$ and $E$ is the set of Definition \ref{d:big}, then,
for each $(f,g) \in E$, $\im g$ is a direct summand of $B$ isomorphic
to $A$. That is, $B$ contains many direct summands isomorphic to
$A$. In view of the preceding result, a natural question arises: is
$B$ isomorphic to a direct sum of copies of $A$? The following example
shows that the answer to this question is negative in general.

\begin{example}
  Let $\kappa$ be an infinite regular cardinal and consider the
  abelian group $M=\mathbb Z^{(\kappa)} \oplus \frac{\mathbb Z}{2
    \mathbb Z}$. Then $\mathbb Z \prec M$ by Proposition
  \ref{p:ExistenceBigObjectsGrothendieck} and Lemma
  \ref{l:PropertiesOfBigObjects}, while $M$ is not isomorphic to a
  direct sum of copies of $\mathbb Z$ since it is not free. 
\end{example}

\section{MAXIMAL IDEALS}
\label{sec:maximal-ideals}

In this section, using the order $\preceq$, we prove that there do not exist maximal ideals in
Grothendieck categories. We will prove a more general result: if
$\mathbf C$ is a preadditive category such that there is no
non-zero maximal object with respect to $\preceq$, then
$\mathbf C$ does not have maximal ideals. The proof is based on the
following theorem:

\begin{theorem}\label{t:MainTheorem}
  Let $\mathbf C$ be a preadditive category, $A$ and $B$ objects of
  $\mathbf C$ such that $A \prec B$, and $I$ a proper
  ideal of $\End_{\mathbf C}(A)$. Then $\mathcal A_I(B,B)$ is a proper
  ideal of $\End_{\mathbf C}(B)$ which is not maximal.
\end{theorem}

\begin{proof}
  Since $A \prec B$, there exists
  $E \subseteq \Hom_{\mathbf C}(B,A) \times \Hom_{\mathbf C}(A,B)$
  satisfying the conditions of Definition \ref{d:big}. Let $J$ be the
  ideal of $\End_{\mathbf C}(B)$ generated by the set of all the
  endomorphisms of $B$ that factors through $A$. We claim that
  $J+\mathcal A_I(B,B)$ is a proper ideal of $\End_{\mathbf C}(B)$
  strictly containing $\mathcal A_I(B,B)$.
  
  First of all, note that $\mathcal A_I(B,B)$ is not equal to
  $J+\mathcal A_I(B,B)$. In fact, for each $(f,g) \in E$, $g f$ is an
  element of $J$ not belonging to $\mathcal A_I(B,B)$, since
  $f g f g = 1_A\notin I$ because $I$ is proper.

  Now we will prove that $J+\mathcal A_I(B,B)$ is a proper ideal. Fix
  any element $\psi \in J+\mathcal A_I(B,B)$. Let $\varphi \in J$ and
  $\varphi' \in \mathcal A_I(B,B)$ be such that
  $\psi=\varphi+\varphi'$. Since $\varphi \in J$,
  $\varphi=\sum_{i=1}^nf_ig_i$ for morphisms
  $f_1, \ldots, f_n \in \Hom_{\mathbf C}(A,B)$ and
  $g_1, \ldots, g_n \in \Hom_{\mathbf C}(B,A)$. The set
  $\{(f,g) \in E:f \psi g \notin I\}$ is contained in the set
  $\{(f,g) \in E:f \varphi \neq 0\}$, which is contained in
  \begin{displaymath}
    \bigcup_{i=1}^n \{(f,g) \in E:ff_i \neq 0\}.
  \end{displaymath}
  Since $A \prec B$ and $E$ is infinite, this set has
  cardinality smaller than $|E|$. The conclusion is that, for each
  $\psi \in J+\mathcal A_I(B,B)$, the set
  $\{(f,g) \in E: f\psi g \notin I\}$ has cardinality smaller than
  $|E|$. But this implies that $1_{B}$ does not belong to
  $J+\mathcal A_I(B,B)$, as $\{(f,g) \in E:f 1_{B} g \notin I\} =
  E$. Consequently, $J+\mathcal A_I(B,B)$ is a proper ideal.
\end{proof}

This theorem has a number of consequences.

\begin{corollary}\label{c:MaximalInPreadditive}
  Let $\mathbf C$ be a preadditive category such that there do not
  exist non-zero maximal objects with respect to $\preceq$. Then
  $\mathbf C$ does not have maximal ideals.
\end{corollary}

\begin{proof}
  Let $\mathcal I$ be any proper ideal in $\mathbf C$. Then there
  exists an object $A$ such that
  $\mathcal I(A,A) \neq \End_{\mathbf C}(A)$. Set $I=\mathcal
  I(A,A)$.
  Let $B$ be an object such that $A \prec B$. Then
  $\mathcal I(B,B) \subseteq \mathcal A_I(B,B)$ which, as a
  consequence of the previous result, is properly contained in a
  proper ideal of $\End_{\mathbf C}(B)$. This means that
  $\mathcal I(B,B)$ is not a maximal ideal of $\End_{\mathbf C}(B)$,
  and $\mathcal I$ is not a maximal ideal in $\mathbf C$ by
  \cite[Lemma 2.4]{FacchiniPerone}.
\end{proof}

Combining this result with Proposition
\ref{p:ExistenceBigObjectsGrothendieck}, we obtain that maximal
ideals do not exist in any
Grothendieck category (in particular, in any module category).

\begin{corollary}
  Let $\mathbf G$ be a Grothendieck category. Then there do not exist
  maximal ideals in $\mathbf G$.
\end{corollary}

Another remarkable consequence of Theorem \ref{t:MainTheorem} is the
following.

\begin{corollary}\label{c:MaximalIdealsInBigObjects}
  Let $\mathbf C$ be a preadditive category, $\mathcal M$ a maximal
  ideal of $\mathbf C$ and $A$ an object of $\mathbf C$. If
  $\mathcal M(A,A) \neq
  \End_R(A)$, then $A$ is maximal with respect to $\preceq$.
\end{corollary}

\begin{proof}
  Set $I=\mathcal M(A,A)$. Suppose that there exists an object $B$
  such that $A \prec B$. Then $\mathcal M(B,B) = \mathcal A_I(B,B)$ by
  \cite[Lemma 2.4]{FacchiniPerone}. But, by Theorem~\ref{t:MainTheorem},
  $\mathcal A_I(B,B)$ is a proper ideal that is not maximal. This
  contradicts the maximality of $\mathcal M$.
\end{proof}

\begin{remark} \label{r:MaximalInBigObjects} {\rm Notice that,
    equivalently, if $A$ is an object of a preadditive category
    $\mathbf C$ and there exists an object $B$ with $A\prec B$, then
    $\mathcal M(A,A) = \End_\mathbf{C}(A)$ for every maximal ideal
    $\mathcal M$ in $\mathbf C$.}
\end{remark}

As a consequence of Corollary \ref{c:MaximalInPreadditive}, if a
preadditive category $\mathbf C$ has maximal ideals, then there exist
maximal objects with respect to the partial order $\preceq$. However, not all objects have to be
maximal. That is, there can exist objects $A$ for which there are 
objects $B$ with $A\prec B$. For example, let $\kappa$ a cardinal and
$\kappa^+$ be its successor cardinal. Let $K$ be a field and $\mathbf C$
the full subcategory of Mod-$K$ whose objects are all vector spaces
of dimension smaller than $\kappa^+$. As is proved in \cite[Example
4.1]{FacchiniPerone}, $\mathbf C$ has one maximal ideal. However, for
each vector space $M$ of dimension smaller than $\kappa$, there
exists spaces $V$ with $M\prec V$ by Proposition
\ref{p:BigVectorSpaces}.

Despite this observation, we are going to see that, in order to
determine if a preadditive category has maximal ideals, we can
restrict our attention to a full subcategory in which each object is
maximal with respect to $\preceq$.

\begin{definition}
  Let $\mathbf C$ be a preadditive category.
  \begin{enumerate}
  \item We will denote by $\mathbf M(\mathbf C)$ the full subcategory
    of $\mathbf C$ consisiting of all maximal objects with respec to
    $\preceq$, that is,
    \begin{displaymath}
      \{\,C \in \mathbf C: \textrm{ there does not exist } A \in \mathbf{C} \textrm{ with }C\prec A\,\}
    \end{displaymath}

  \item We will denote by $\mathbf S(\mathbf C)$ the full subcategory
    of $\mathbf C$ whose class of objects is
    \begin{displaymath}
      \{C \in \mathbf C: \textrm{there exists } A \in \mathbf
      M(\mathbf{C}) \textrm{ with }C \prec A\}
    \end{displaymath}
  \end{enumerate}
\end{definition}

Let $\mathbf C$ be a preadditive category and $\mathbf D$ a full
subcategory of $\mathbf C$. We now define how to restrict an ideal of
$\mathbf C$ to $\mathbf D$ and, conversely, how to extend a maximal
ideal of $\mathbf D$ to $\mathbf C$.

\begin{definition}
  Let $\mathbf C$ be a preadditive category and $\mathbf D$ a full
  subcategory of $\mathbf C$.
  \begin{enumerate}
  \item Given $\mathcal I$ an ideal of $\mathbf C$, define its
    restriction $\mathcal I^{\re}$ to $\mathbf D$ by
    \begin{displaymath}
      \mathcal I^{\re}(D,D') = \mathcal I(D,D')
    \end{displaymath}
    for every $D,D' \in \mathbf D$.

  \item Given any maximal ideal $\mathcal M$ of $\mathbf D$, there
    exists an object $D \in \mathbf D$ such that
    $\mathcal M(D,D) \neq \End_{\mathbf D}(D)$. Define the
    extension $\mathcal M^{\ex}$ of $\mathcal M$ to $\mathbf C$ to be
    the ideal of $\mathbf C$ associated to $\mathcal M(D,D)$.
  \end{enumerate}
\end{definition}

\begin{lemma}\label{l:DefinitionExtension}
  Let $\mathbf C$ be a preadditive category, $\mathbf D$ a full
  subcategory of $\mathbf C$ and $\mathcal M$ a maximal ideal in
  $\mathbf D$.
  \begin{enumerate}
  \item Let $D$ and $D'$ be objects of $\mathbf D$ such that
    $I:=\mathcal M(D,D)$ and $I':=\mathcal M(D',D')$ are maximal
    ideals of $\End_{\mathbf C}(D)$ and $\End_{\mathbf C}(D')$
    respectively. Then the ideals $\mathcal A_I$ and $\mathcal A_{I'}$
    coincide in $\mathbf C$. In particular, the definition of
    $\mathcal M^{\ex}$ does not depend on the choice of the object $D$
    with $\mathcal M(D,D) \neq \End_{\mathbf D}(D)$.

  \item For any objects $D,D'$ of $\mathbf D$,
    $\mathcal M^{\ex}(D,D') = \mathcal M(D,D')$.
  \end{enumerate}

\end{lemma}

\begin{proof}
  {\em (1)} By \cite[Lemma 2.4]{FacchiniPerone},
  $\mathcal M=\mathcal A_I$ in $\mathbf D$. Then
  $\mathcal A_I(D',D') \subseteq I'$, which implies that
  $\mathcal A_I$ is contained in $\mathcal A_{I'}$ (in $\mathbf
  C$).
  Using the same argument, $\mathcal M=\mathcal A_{I'}$ in $\mathbf D$
  and, consequently, $\mathcal A_{I'}(D,D) \subseteq I$. Thus
  $\mathcal A_{I'}$ is contained in $\mathcal A_I$ (in $\mathbf C$).

  {\em (2)} By \cite[Lemma 2.4]{FacchiniPerone}.
\end{proof}

Now we can establish, for any preadditive category $\mathbf C$, the
relation between the maximal ideals of $\mathbf C$ and those of
$\mathbf M(\mathbf C)$.

\begin{theorem}\label{p:MaximalMC}
  Let $\mathbf C$ be a preadditive category. Then the assignments
  $\mathcal M \mapsto \mathcal M^{\re}$ and
  $\mathcal M \mapsto \mathcal M^{\ex}$ define bijective
  correspondences between the following classes of ideals:
  \begin{enumerate}
  \item Maximal ideals of $\mathbf C$.

  \item Maximal ideals $\mathcal M$ of $\mathbf M(\mathbf C)$
    satisfying $\mathcal M^{\ex}(C,C)=\End_{\mathbf C}(C)$ for each
    object $C$ not belonging to
    $\mathbf M(\mathbf C) \cup \mathbf S(\mathbf C)$, that is, for each
    object $C$ with no maximal $N$ with $C\preceq N$.
  \end{enumerate}
\end{theorem}

\begin{proof}
  Let $\mathcal M$ be a maximal ideal of $\mathbf C$. We will now
  prove that $\mathcal M^{\re}$ is a maximal ideal of
  $\mathbf M(\mathbf C)$ satisfying
  $\mathcal M^{\re\ex}(C,C)=\End_{\mathbf C}(C)$ for each object $C$
  not belonging to $\mathbf M(\mathbf C) \cup \mathbf S(\mathbf
  C)$.
  Since $\mathcal M$ is proper, there exists an object $C_0$ of
  $\mathbf C$ such that
  $\mathcal M(C_0,C_0) \neq \End_{\mathbf C}(C_0)$. By Theorem
  \ref{t:MainTheorem}, $C_0$ must belong to $\mathbf M(\mathbf
  C)$.
  This means that $\mathcal M^{\re}$ is a proper ideal in
  $\mathbf M(\mathbf C)$, which is trivially maximal, as $\mathcal M$
  is maximal in $\mathbf C$. Moreover, note that $\mathcal M^{\re\ex}$
  is the ideal of $\mathbf C$ associated to $\mathcal M(C_0,C_0)$,
  which is equal to $\mathcal M$ by \cite[Lemma
  2.4]{FacchiniPerone}. Then, again by Theorem \ref{t:MainTheorem},
  $\mathcal M^{\re\ex}(C,C) = \mathcal M(C,C) = \End_{\mathbf C}(C)$
  for every object $C$ not belonging to $\mathbf M(\mathbf C)$.

  Conversely, let $\mathcal M$ be a maximal ideal of
  $\mathbf M(\mathbf C)$ satisfying
  $\mathcal M^{\ex}(C,C) = \End_{\mathbf C}(C)$ for each object $C$
  not belonging to $\mathbf M(\mathbf C) \cup \mathbf S(\mathbf
  C)$.
  We claim that $\mathcal M^{\ex}(C,C)=\End_{\mathbf C}(C)$ for each
  object $C$ belonging to $\mathbf S(\mathbf C)$. To prove the claim,
  let $C $ be an object of $\mathbf S(\mathbf C)$ and suppose
  that $D \in \mathbf M(\mathbf C)$ satisfies $C \prec D$. If
  $\mathcal M^{\ex}(C,C) \neq \End_{\mathbf C}(C)$, then
  $\mathcal M^{\ex}(D,D)$ is a proper ideal of $\End_{\mathbf C}(D)$,
  which is not maximal by Theorem \ref{t:MainTheorem}. By Lemma
  \ref{l:DefinitionExtension},
  $\mathcal M^{\ex}(D,D) = \mathcal M(D,D)$ and, consequently,
  $\mathcal M(D,D)$ is a proper ideal of
  $\End_{\mathbf M (\mathbf C)}(D)$ that is not maximal. Since
  $\mathcal M$ is maximal in $\mathbf M(\mathbf C)$, this contadicts
  \cite[Lemma 2.4]{FacchiniPerone}. The contradiction proves the
  claim.

  Now let $\mathcal N$ be an ideal of $\mathbf C$ properly containing
  $\mathcal M^{\ex}$. We will prove that
  $\mathcal N=\Hom_{\mathbf C}$. Since
  $\mathcal N(C,C)=\End_{\mathbf C}(C)$ for each object $C$ not
  belonging to $\mathbf M(\mathbf C)$, it follows that
  $\mathcal N^{\re}$ properly contains $\mathcal M^{\ex\re}$. But
  $\mathcal M^{\ex\re}=\mathcal M$ by Lemma
  \ref{l:DefinitionExtension} and, since $\mathcal M$ is maximal in
  $\mathbf M(\mathbf C)$, we get that
  $\mathcal N^{\re}=\Hom_{\mathbf M(\mathbf C)}$. This fact with the
  previous claim gives that $\mathcal N = \Hom_{\mathbf C}$.

  Finally, it is easy to see that the two assignments are mutually
  inverse.
\end{proof}

\begin{remark}{\rm
  Let $\mathbf C$ be a preadditive category, $A$ and $C$ objects of
  $\mathbf C$ and $I$ an ideal of $\End_{\mathbf C}(C)$. Then,
  $\mathcal A_I(A,A) = \End_{\mathbf C}(A)$ if and only if each
  endomorphism of $C$ factoring through $A$ belongs to $I$.
  Consequently, if $\mathcal M$ is a maximal ideal in
  $\mathbf M(\mathbf C)$ and $C$ is an object with no maximal $N$
  satisfying $C \preceq N$, then the following
  conditions are equivalent:
  \begin{enumerate}
  \item $\mathcal M^{\ex}(C,C) = \End_{\mathbf C}(C)$.

  \item There exists an object $A \in \mathbf M(\mathbf C)$ with
    $\mathcal M(A,A) \neq \End_{\mathbf C}(A)$ such that each
    endomorphism of $A$ factoring through $C$ belongs to
    $\mathcal M(A,A)$.

  \item For each object $A \in \mathbf M(\mathbf C)$ with
    $\mathcal M(A,A) \neq \End_{\mathbf C}(A)$, every endomorphism of $A$ factoring through $C$ belongs to
    $\mathcal M(A,A)$.
  \end{enumerate}}
\end{remark}

Proposition \ref{p:MaximalMC} says that, in order to compute the
maximal ideals in a category $\mathbf C$, we can {\em (1)} determine
the subcategories $\mathbf M(\mathbf C)$ and $\mathbf S(\mathbf C)$,
and {\em (2)} find the maximal ideals $\mathcal M$ of
$\mathbf M(\mathbf C)$ such that
$\mathcal M^{\ex}(C,C)=\End_{\mathbf C}(C)$ for each object $C$ with
no maximal $N$ satisfying $C \preceq N$. We will
use this procedure in the following example.

\begin{example}{\rm Let $R$ be a simple non-artinian ring with
    $\Soc(R_R)$ non-projective as a right $R$-module. Then there
    exists a non projective simple right module $S$ contained in
    $R$. Consider the full subcategory
    $\mathbf C=\add(R_R) \cup \Add(S)$ of $\Modr$. Then
    $\mathbf M(\mathbf C) = \add(R_R)$ and
    $\mathbf S(\mathbf C)=\emptyset$. Since $R$ is simple, Example
    \ref{e:MaximalIdeals} says that the unique maximal ideal of
    $\mathbf M(\mathbf C)$ is the ideal $\mathcal A_0$ associated to
    the zero ideal of $R$. However, $\mathcal A_0(S,S) \neq \End_R(S)$
    since, if we take $f:R_R \rightarrow S$ an epimorphism and we
    denote by $g:S \rightarrow R_R$ the inclusion, we have that
    $g1_Sf \neq 0$, which means that $1_S \notin \mathcal
    A_0(S,S)$.
    Then, by Theorem \ref{p:MaximalMC}, $\mathbf C$ does not have
    maximal ideals.}
\end{example}

\begin{remark}{\rm Let $\mathbf C$ be a preadditive category. As the
    preceding example shows, there does not exist a bijective
    correspondence between maximal ideals in $\mathbf C$ and maximal
    ideals in $\mathbf M(\mathbf C)$. This is due to the fact that
    there can exist maximal ideals $\mathcal M$ of
    $\mathbf M(\mathbf C)$ such that
    $\mathcal M^{\ex}(C,C) \neq \End_{\mathbf C}(C)$ for some object
    $C$ with no maximal $N$ satisfying $C \preceq N$. More precisely,
    the map $(-)^{\re}$ from the class of maximal ideals of $\mathbf
    C$ to the class of maximal ideals of $\mathbf M(\mathbf C)$ is
    injective, and its image consists of those maximal ideals
    $\mathcal M$ of $\mathbf M(\mathbf C)$ for which $\mathcal
    M^{\re}(C,C) = \End_{\mathbf C}(C)$ for each object $C$ with no
    maximal object $M$ with $C\preceq M$. The map $(-)^{\re}$ is not,
    in general, surjective .}
\end{remark}

\section{MAXIMAL IDEALS INDUCED BY AN INDECOMPOSABLE INJECTIVE MODULE}
\label{sec:maxim-ideals-induc}

We conclude the paper computing the maximal ideals of a certain
subcategory of a module category. The main idea in this computation is
to apply the results of the previous sections to describe the
ideal associated to a maximal ideal in the endomorphism ring of an
indecomposable injective module.

Let $R$ be a ring that we fix through the rest of the section and let
$F$ be an injective $R$-module. Recall that $F$ is indecomposable if and
only if $F$ has a local endomorphism ring \cite[Theorem
25.4]{AndersonFuller}.

\begin{lemma}
  Let $F$ be an indecomposable injective module, and let $I$ be the
  maximal ideal of $\End_R(F)$. The following conditions are
  equivalent for modules $A$, $B$ and $f \in \Hom_{R}(A,B)$:
  \begin{enumerate}
  \item $f \notin \mathcal A_I(A,B)$.

  \item There exists $\alpha\colon F \rightarrow A$ and
    $\beta\colon B \rightarrow F$ such that $\beta f \alpha = 1_F$.

  \item There exists a submodule $C$ of $A$ such that $C \cong F$ and
    $C \cap \ker f =0$.
  \end{enumerate}
\end{lemma}

\begin{proof}
  {\em (1)} $\Rightarrow$ {\em (2)}. If $f \notin \mathcal A_I$, there
  exists $\alpha\colon F \rightarrow A$ and
  $\beta\colon A \rightarrow F$ such that $\beta f \alpha \notin I$.
  Since the endomorphisms of $F$ not belonging to $I$ are isomorphisms,
  there exists an inverse $\gamma \in \End_R(F)$ of $\beta f
  \alpha$. Then $\gamma \beta f \alpha = 1_F$.

 {\em (2)} $\Rightarrow$ {\em (1)} is trivial.
 
  {\em (2)}${}\Rightarrow{}${\em (3)}. Set $C = \alpha(F)$, which is
  isomorphic to $F$ as $\alpha$ is monic. Then $A=C\oplus\ker(\beta f)$. In particular, $C \cap \ker f \subseteq C \cap \ker(\beta f )=0$.

  {\em (3)}${}\Rightarrow{}${\em (2)}. Let
  $\alpha\colon F \rightarrow A$ be a monomorphism with image
  $C$. Since $C \cap \ker f = 0$, $f\alpha$ is a monomorphism. Since $F$ is an injective, this implies the existence of $\beta\colon B\to F$ with $\beta f \alpha = 1_F$, as desired.\end{proof}

As a byproduct of this result we get:

\begin{corollary}\label{c:IdealAI}
  Let $F$ be an indecomposable injective module, and let $I$ be the
  maximal ideal of $\End_R(F)$. Then, for any pair $A,B$ of modules, $\mathcal A_I(A,B) \neq \Hom_R(A,B)$ if and only if both $A$
  and $B$ contain a submodule isomorphic to $F$.
\end{corollary}

\begin{proof}
  If $\mathcal A_I(A,B) \neq \Hom_R(A,B)$ and
  $f \colon A \rightarrow B$ does not belong to $\mathcal A_I(A,B)$,
  then, by the previous lemma, there exists $C \leq A$ with
  $C \cong F$ and $C \cap \ker f =0$. This implies that $f(C)$ is
  isomorphic to $F$. Thus $C$ and $f(C)$ are submodules with the desired property.

  Conversely, assume $A=A_1\oplus A_2$ and $B=B_1 \oplus B_2$ with
  $A_1 \cong B_1 \cong F$. Let $f:A_1 \rightarrow B_1$ be an
  isomorphism and let $F:A \rightarrow B$ be the morphism
  $f \oplus 0$. Then $F$ trivially satisfies {\em (3)} of the
  previous lemma and, consequently, $F$ does not belong to
  $\mathcal A_I(A,B)$.
\end{proof}

It follows that, in order to determine
$\mathcal A_I$, we only have to look at the modules $A$ containing
isomorphic copies of $F$. These modules have a nice description if $R$
is right noetherian as we will prove next. We shall use the following
well known facts about an injective module $F$:
\begin{enumerate}
\item If $K$ is a non-zero submodule of $F$, then there exists an
  injective submodule $G$ of $F$ containing $K$ such that the
  inclusion $K \leq G$ is an injective envelope. In particular, if $F$
  is indecomposable, then the inclusion $K \leq F$ is an injective
  envelope of $K$.

\item $F$ satisfies the exchange property, which means that for each
  module $N$ and each decomposition
  $F \oplus N = \bigoplus_{\alpha < \kappa}A_\alpha$, there exists a
  submodule $B_\alpha \leq A_\alpha$ for each $\alpha < \kappa$ such
  that
  $F \oplus N =F \oplus\left( \bigoplus_{\alpha <
      \kappa}B_\alpha\right)$.
\end{enumerate}

\begin{theorem}\label{t:Decompositions}
  Suppose $R$ right noetherian. Let $F$ be an indecomposable injective
  module. Then every module $M$ has a decomposition $M=M_1 \oplus M_2$
  where:
  \begin{enumerate}
  \item  $M_1 \cong F^{(\Gamma)}$ for some set $\Gamma$.
  \item $M_2$ does not contain submodules isomorphic to $F$.
  \end{enumerate}
  Moreover, $\Gamma$ is uniquely determined up to cardinality and
  $M_2$ is uniquely determined up to isomorphism.
\end{theorem}

\begin{proof}
  If $M$ does not have submodules isomorphic to $F$, there is nothing
  to prove. So suppose that $M$ has submodules isomorphic to $F$ and
  consider the  non-empty family of submodules
  \begin{displaymath}
    \mathcal S = \{N \leq M: N
    \cong F^{(\Gamma)} \textrm{ for some set }\Gamma\}
  \end{displaymath}
  Let us show that $\mathcal S$ is inductive. Take a chain
  $\mathcal S'=\{\,M_\lambda: \lambda\in\Lambda\,\}$ in $\mathcal S$.
  We will prove that
  $M':=\bigcup_{\lambda\in\Lambda}M_\lambda \in \mathcal S$. Since $R$
  is right noetherian, directed colimits of injective modules are
  injective \cite[Exercise 8 of Chapter I]{CartanEilenberg}, so that $M'$ is
  injective. Hence, $M'$ has a direct-sum decomposition,
  $M'=\bigoplus_{i\in I}F_i$, where the submodules $F_i$ of $M'$ are
  injective and indecomposable \cite[Theorem~3.48, p.~82]{Lam}. Let
  $i \in I$ and $x$ be a non-zero element of $F_i$; note that $F_i$ is
  the injective envelope of $xR$. Since $x\in M_\lambda$ for some
  $\lambda\in\Lambda$, and $M_\lambda \in \mathcal S$, $x$ belongs to
  a direct summand of $M_\lambda$ isomorphic to $F^n$ for some
  $n$. But, as $F_i$ is the injective envelope of $xR$, $F^n$ must
  contain a direct summand isomorphic to $F_i$. This implies that
  $F_i \cong F$ because all indecomposable direct summands of $F^n$
  are isomorphic to $F$. Consequently, $M' \in \mathcal S$.
  
  The first part of the statement now follows taking a maximal element
  $M_1$ of $\mathcal S$ and a submodule $M_2 $ of $M$ with
  $M_1 \oplus M_2 = M$.

  In order to prove the last part of the statement, suppose that
  $M = M'_1 \oplus M'_2$ is another decomposition of $M$ satisfying
  {\em (1)} and {\em (2)}. Write
  $M_1 = \bigoplus_{\beta < \kappa}G_\beta$ and
  $M'_1 = \bigoplus_{\alpha < \lambda} F_\alpha$ for suitable families
  of submodules $\{G_\beta:\beta < \kappa\}$ and
  $\{F_\alpha:\alpha < \lambda\}$ of $M_1$ and $M'_1$ respectively,
  and cardinals $\kappa$ and $\lambda$, satisfying
  $G_\beta \cong F_\alpha \cong F$ for each $\beta < \kappa$ and
  $\alpha < \lambda$.

  Since $M_1$ satisfies the exchange property, there exist submodules
  $H_\alpha \leq F_\alpha$ for each $\alpha < \lambda$ and
  $N'_2 \leq M'_2$ such that
  $M = M_1 \oplus \left(\bigoplus_{\alpha < \lambda}H_\alpha\right)
  \oplus N'_2$.
  Since $F_\alpha$ is indecomposable for each $\alpha < \lambda$, it
  follows that $H_\alpha=0$ or $H_\alpha = F_\alpha$. But, as
  $\left(\bigoplus_{\alpha < \lambda}H_\alpha\right) \oplus N'_2 \cong
  M_2$
  and $M_2$ does not contain any submodule isomorphic to $F$, we get that $H_\alpha = 0$ for each $\alpha < \lambda$. That is,
  $M=M_1 \oplus N'_2$.

  Applying the modular law, we see that
  $M'_2 = N'_2 \oplus (M'_2 \cap M_1)$. We claim that
  $M'_2 \cap M_1=0$. Assume the contrary, i.~e., that
  $M_1 \cap M'_2 \neq 0$. Since $M_1 \cap M'_2$ is a direct summand of
  $M$, it is a direct summand of $M_1$ and, consequently, it is
  injective. As it is non-zero, there exists $\beta < \kappa$ such
  that $G_\beta \cap M_1 \cap M'_2 \neq 0$. Let $x$ be a non-zero
  element in this intersection. Notice that $xR \leq G_\beta$ is an
  injective envelope. Since $M_1 \cap M'_2$ is injective, there
  exists an injective envelope $C$ of $xR$ contained in $M_1 \cap M'_2$.
  But $C$ is isomorphic to $F$ and $M'_2$ does not contain any
  submodule isomorphic to $F$, a contradiction. This proves the claim.

  As a consequence, $M=M_1 \oplus M'_2$. Then
  $M_1 \cong M'_1$ and $M_2 \cong M'_2$. By Azumaya's Theorem
  \cite[Theorem 12.6]{AndersonFuller}, $\kappa = \lambda$ and we are
  done.
\end{proof}

We can use this result to define the $F$-rank of a module $M$, for any
indecomposable injective module $F$ and any module $M$ over a right
noetherian ring.

\begin{definition}
  Suppose $R$ is right noetherian. Let $F$ be an indecomposable
  injective module. Given any module $M$ and any cardinal $\kappa$, we
  say that {\em $M$ has $F$-rank equal to $\kappa$} (written
  $\re_F(M) = \kappa$) if $M = M_1 \oplus M_2$, where $M_1 \cong F^{(\kappa)}$ and $M_2$ has no direct summand isomorphic to $F$.
  We will denote by $\mathbf C_F$ the full subcategory of $\Modr$
  whose objects are all modules of finite $F$-rank.
\end{definition}

Now we can compute the maximal ideals in the category $\mathbf C_F$
for an indecomposable injective module $F$. First of all, we compute
the subcategories $\mathbf M(\mathbf C_F)$ and
$\mathbf S(\mathbf C_F)$.

\begin{proposition}\label{p:SandM}
  Suppose $R$ is right noetherian. Let $F$ be an indecomposable
  injective module. Then:
  \begin{enumerate}
  \item $\mathbf M(\mathbf C_F) = \{M \in \mathbf C_F:\re_F(M)>0\}$.

  \item $\mathbf S(\mathbf C_F) =\{M \in \mathbf C_F:\re_F(M)=0\}$.
  \end{enumerate}
\end{proposition}

\begin{proof}
  {\em (1)} Let $M$ be a module in $\mathbf C_F$ with $\re_F(M)=0$. We can find an infinite cardinal $\kappa$
  such that $M$ is $< \kappa$-generated in $\Modr$. By Proposition
  \ref{p:ExistenceBigObjectsGrothendieck}, $M \prec M^{(\kappa)}$ in $\Modr$. Since $\mathbf C_F$ is a full
  subcategory of $\Modr$ and $M^{(\kappa)} \in \mathbf C_F$, we get that 
  $M \prec M^{(\kappa)}$ in $\mathbf C_F$.
  Consequently, $M$ does not belong to $\mathbf M(\mathbf C_F)$. This
  proves the inclusion
  $\mathbf M(\mathbf C_F) \subseteq \{M \in \mathbf C_F:\re_F(M)>0\}$.

  In order to prove the inverse inclusion, let $M$ be any module with
  $\re_F(M)>0$ and suppose, by contradiction, that
  $M \notin \mathbf M(\mathbf C_F)$. Then there exists
  $N \in \mathbf C_F$ such that $M \prec N$. Let $M=M_1 \oplus M_2$
  and $N=N_1 \oplus N_2$ be the decompositions given by Theorem
  \ref{t:Decompositions}, and let $E$ be the set of Definition
  \ref{d:big}. Write $M_1=\bigoplus_{i=1}^nG_i$ and
  $N_1=\bigoplus_{j=1}^mF_j$ for modules $G_i$ and $F_j$ isomorphic to
  $F$ for each $i =1, \ldots, n$ and $j=1, \ldots, m$.

  We claim that $f(N_1) \neq 0$ for each $(f,g) \in E$. Given
  $i =1, \ldots, n$, $g(G_i)$ is isomorphic to $F$. Then
  $g(E_i)\cap N_1 \neq 0$ since, otherwise, $N_1 \oplus g(E_i)$ would
  be a direct summand of $N$, and $N_2$ would contain a submodule
  isomorphic to $F$. Now, taking $y \in g(E_i) \cap N_1$ non-zero and
  $x \in E_i$ with $g(x)=y$, we have that $f(y)=x \neq 0$. This proves
  the claim.

  For each $i =1, \ldots, m$, let $q_i:G_1 \rightarrow F_i$ be
  an isomorphism. Then $q_i$ extends to a morphism $p_i\colon M \rightarrow N$.
  The preceeding claim says that for each $(f,g) \in E$, $fp_i\neq 0$
  for some $i =1, \ldots, m$. Consequently, if   \begin{displaymath}
    E_i = \{(f,g) \in E:fp_i \neq 0\}
  \end{displaymath}
  for each $i =1, \ldots, m$, we conclude that
  $E \subseteq \bigcup_{i=1}^mE_i$. This is a contradiction, because
  the second set has cardinality smaller than $|E|$ by 
  Definition \ref{d:big}{\em (2)}.

  The conclusion is that there is no object $N \in \mathbf C_F$ with $M \prec N$, so that $M \in \mathbf M(\mathbf C_F)$.

  {\em (2)} As a direct consequence of {\em (1)} we have
  $\mathbf S(\mathbf C_F) \subseteq \{M \in \mathbf C_F:\re_F(M)=0\}$.
  In order to see the other inclusion, fix $M \in \mathbf C_F$ with
  $\re_F(M)=0$. Then, as in the proof of {\em (1)}, there exists an
  infinite cardinal $\kappa$ such that $M\prec M^{(\kappa)}$ in $\mathbf C_F$. By Lemma
  \ref{l:IdealsInDirectSums}, $M\prec F \oplus M^{(\kappa)}$. Then $M \in \mathbf S(\mathbf C_F)$ because
  $F \oplus M^{(\kappa)} \in \mathbf M(\mathbf C_F)$ by~{\em (1)}.
\end{proof}

Finally, we can determine all maximal ideals of the category
$\mathbf C_F$ for an indecomposable injective module $F$ over a right
noetherian ring.

\begin{proposition}
  Suppose $R$ is right noetherian. Let $F$ be an indecomposable
  injective module and $I$ be the maximal ideal of $\End_R(F)$. Then
  $\mathcal A_I$ is the unique maximal ideal of $\mathbf C_F$.
\end{proposition}

\begin{proof}
  By Theorem \ref{p:MaximalMC}, we only have to compute the maximal
  ideals of $\mathbf M(\mathbf C_F)$. First, we prove that
  $\mathcal A_I$ is a maximal ideal of $\mathbf M(\mathbf C_F)$. Given
  $M \in \mathbf M(\mathbf C_F)$, since
  $\mathcal A_I(M,M) \neq \End_R(M)$ by Corollary
  \ref{c:IdealAI}, we have to see, applying \cite[Lemma
  2.4]{FacchiniPerone}, that
  \begin{enumerate}
  \item[{\em (a)}] $\mathcal A_I(M,M)$ is maximal in $\End_R(M)$ and,

  \item[{\em (b)}] if $J_0=\mathcal A_I(M,M)$, then
    $\mathcal A_I = \mathcal A_{J_0}$.
  \end{enumerate}
  Let $M=M_1 \oplus M_2$ be the decomposition of $M$ given in Theorem
  \ref{t:Decompositions}. Note that, by Lemma
  \ref{l:IdealsInDirectSums} and Corollary \ref{c:IdealAI},
  \begin{displaymath}
    \mathcal A_I(M,M) = \{f \in \End_R(M):\pi_1 f \iota_1 \in \mathcal
    A_I(M_1,M_1)\},
  \end{displaymath}
  where $\pi_i:M \rightarrow M_i$ and $\iota_i:M_i \rightarrow M$ are
  the corresponding projections and inclusions for $i =1,2$.
  As $M_1 \in \add(F)$ and $\mathcal A_I$ is a maximal ideal
  in this category by Example \ref{e:MaximalIdeals},
  $\mathcal A_I(M_1,M_1)$ is a maximal ideal in $\End_R(M_1,M_1)$. In
  order to see that $\mathcal A_I(M,M)$ is maximal, let $J$ be an
  ideal of $\End_R(M)$ strictly containing $\mathcal A_I(M,M)$. Let $f \in J$ not belonging to $\mathcal A_I(M,M)$. Then
  $\pi_1 f \iota_1$ does not belong to $\mathcal A_I(M_1,M_1)$ and,
  by the maximality of this ideal in $\End_R(M_1,M_1)$, there exist
  $g \in \mathcal A_I(M_1,M_1)$ and $\alpha, \beta \in \End_R(M_1)$
  such that $1_{M_1} = g+\alpha \pi_1 f \iota_1 \beta$. Then we have
  the identity  \begin{displaymath}
    1_{M_1} \oplus 0 = g \oplus 0 + (\alpha \oplus 0) f (\beta \oplus 0)
  \end{displaymath} in $\End_R(M)$,
  with both $g$ and $(\alpha \oplus 0) f (\beta \oplus 0)$ in
  $J$. Consequently, $1_{M_1} \oplus 0 \in J$. Now use
  $0 \oplus 1_{M_2} \in J$ to get that
  $1_M = 1_{M_1}\oplus 0+0 \oplus 1_{M_2}\in J$ and that
  $J=\End_R(M)$.

  Let us prove {\em (b)}. Since $\mathcal A_{J_0}$ is the greatest of
  all the ideals $\mathcal I'$ of $\mathbf C_F$ such that
  $\mathcal I'(M,M) \leq J_0$, we conclude that
  $\mathcal A_I \subseteq \mathcal A_{J_0}$. In order to prove the
  other inclusion, we only have to see, by the same argument, that
  $\mathcal A_{J_0}(F,F) \leq I$. Let $f \in \mathcal
  A_{J_0}(F,F)$. Fix a
  monomorphism  $\alpha_1\colon F \rightarrow M_1$, which, as $\im \alpha_1$ is a direct summand, has an splitting
  $\beta_1:M_1 \rightarrow F$. Then
  note that $\iota_1\alpha_1 f \beta_1 \pi_1 \in J_0$, because
  $f \in \mathcal A_{J_0}(F,F)$. Then
  $\pi_1\iota_1 \alpha_1 f \beta_1 \iota_1 \alpha_1 \in \mathcal
  A_I(M_1,M_1)$
  and, consequently,
  $\beta_1\pi_1\iota_1 \alpha_1 f \beta_1 \iota_1 \alpha_1 \in
  \mathcal A_I(F,F) = I$. Since 
  \begin{displaymath}
    f = \beta_1\pi_1\iota_1 \alpha_1 f \beta_1 \iota_1 \alpha_1,
  \end{displaymath}
  we conclude that $f \in I$.

  To finish the proof, we will see that $\mathcal A_I$ is the unique
  maximal ideal of $\mathbf M(\mathbf C_F)$. Let $\mathcal M$ be any
  maximal ideal of $\mathbf M(\mathbf C_F)$ and
  $M \in \mathbf M(\mathbf C_F)$ be such that
  $\mathcal M(M,M) \neq \End_R(M)$. If $J=\mathcal M(M,M)$, then
  $\mathcal M=\mathcal A_J$ by \cite[Lemma 2.4]{FacchiniPerone}. Let
  $M=M_1 \oplus M_2$ be the decomposition of $M$ given by Theorem
  \ref{t:Decompositions}. By Lemma \ref{l:IdealsInDirectSums}, either 
  $\mathcal M(M_1,M_1)$ or $\mathcal M(M_2,M_2)$ have to be proper. But
  $M_2 \in \mathbf S(\mathbf C_F)$ by Proposition \ref{p:SandM}, so
  that $\mathcal M(M_2,M_2) = \End_R(M_2)$ by Remark
  \ref{r:MaximalInBigObjects}. Thus 
  $\mathcal M(M_1,M_1) \neq \End_R(M_1)$ which implies, again by Lemma
  \ref{l:IdealsInDirectSums}, that $\mathcal M(F,F) \neq \End_R(F)$.
  Since $I$ is the unique maximal ideal of $\End_R(F)$ and
  $\mathcal M(E,E)$ is maximal, we conclude that $\mathcal M(F,F)=I$.
  Now 
  $\mathcal M=\mathcal A_I$ by \cite[Lemma 2.4]{FacchiniPerone}, which concludes the proof.
\end{proof}

\begin{example}{\rm The category $\mathbf{C}_F$ has maximal ideals and
     objects $M,N$ with $M\prec N$ since, if $M$ is an object in $\mathbf{C}_F$ with
    $F$-rank 0, then each direct sum of copies of $M$ belongs to
    $\mathbf C_F$. By Proposition
    \ref{p:ExistenceBigObjectsGrothendieck}, there exist objects $N$ in
    $\mathbf C_F$ with $M\prec N$.}
\end{example}

\bibliographystyle{alpha} \bibliography{References}

\begin{bibdiv}
\begin{biblist}

\bib{AdamekRosicky}{book}{
      author={Ad\'amek, Jir\'\i},
      author={Rosick\'y, Jir\'\i},
       title={Locally presentable and accessible categories},
      series={London Mathematical Society Lecture Note Series},
   publisher={Cambridge University Press, Cambridge},
        date={1994},
      volume={189},
        ISBN={0-521-42261-2},
         url={http://0-dx.doi.org.almirez.ual.es/10.1017/CBO9780511600579},
      review={\MR{1294136}},
}

\bib{AndersonFuller}{book}{
      author={Anderson, Frank~W.},
      author={Fuller, Kent~R.},
       title={Rings and categories of modules},
     edition={Second},
      series={Graduate Texts in Mathematics},
   publisher={Springer-Verlag, New York},
        date={1992},
      volume={13},
        ISBN={0-387-97845-3},
         url={http://0-dx.doi.org.almirez.ual.es/10.1007/978-1-4612-4418-9},
      review={\MR{1245487}},
}

\bib{CartanEilenberg}{book}{
      author={Cartan, Henri},
      author={Eilenberg, Samuel},
       title={Homological algebra},
      series={Princeton Landmarks in Mathematics},
   publisher={Princeton University Press, Princeton, NJ},
        date={1999},
        ISBN={0-691-04991-2},
        note={With an appendix by David A. Buchsbaum, Reprint of the 1956
  original},
      review={\MR{1731415}},
}

\bib{EnochsJenda}{book}{
      author={Enochs, Edgar~E.},
      author={Jenda, Overtoun M.~G.},
       title={Relative homological algebra. {V}olume 1},
     edition={extended},
      series={De Gruyter Expositions in Mathematics},
   publisher={Walter de Gruyter GmbH \& Co. KG, Berlin},
        date={2011},
      volume={30},
        ISBN={978-3-11-021520-5},
      review={\MR{2857612}},
}

\bib{Facchini2}{incollection}{
      author={Facchini, A.},
       title={Subdirect representations of categories of modules},
        date={2009},
   booktitle={Rings, modules and representations},
      series={Contemp. Math.},
      volume={480},
   publisher={Amer. Math. Soc., Providence, RI},
       pages={139\ndash 151},
         url={http://0-dx.doi.org.almirez.ual.es/10.1090/conm/480/09372},
      review={\MR{2508149}},
}

\bib{FacchiniPerone2}{article}{
      author={Facchini, A.},
      author={Perone, M.},
       title={On some noteworthy pairs of ideals in {${\rm Mod}$}-{$R$}},
        date={2014},
        ISSN={0927-2852},
     journal={Appl. Categ. Structures},
      volume={22},
      number={1},
       pages={147\ndash 167},
         url={http://0-dx.doi.org.almirez.ual.es/10.1007/s10485-012-9292-5},
      review={\MR{3163512}},
}

\bib{FacchiniPrihoda3}{article}{
      author={Facchini, A.},
      author={P\v~r\'\i hoda, P.},
       title={Factor categories and infinite direct sums},
        date={2009},
        ISSN={1306-6048},
     journal={Int. Electron. J. Algebra},
      volume={5},
       pages={135\ndash 168},
      review={\MR{2471385}},
}

\bib{Facchini}{article}{
      author={Facchini, Alberto},
       title={Krull-{S}chmidt fails for serial modules},
        date={1996},
        ISSN={0002-9947},
     journal={Trans. Amer. Math. Soc.},
      volume={348},
      number={11},
       pages={4561\ndash 4575},
  url={http://0-dx.doi.org.almirez.ual.es/10.1090/S0002-9947-96-01740-0},
      review={\MR{1376546}},
}

\bib{FacchiniPerone}{article}{
      author={Facchini, Alberto},
      author={Perone, Marco},
       title={Maximal ideals in preadditive categories and semilocal
  categories},
        date={2011},
        ISSN={0219-4988},
     journal={J. Algebra Appl.},
      volume={10},
      number={1},
       pages={1\ndash 27},
         url={http://0-dx.doi.org.almirez.ual.es/10.1142/S0219498811004458},
      review={\MR{2784751}},
}

\bib{FacchiniPrihoda4}{article}{
      author={Facchini, Alberto},
      author={P\v~r\'\i hoda, Pavel},
       title={Endomorphism rings with finitely many maximal right ideals},
        date={2011},
        ISSN={0092-7872},
     journal={Comm. Algebra},
      volume={39},
      number={9},
       pages={3317\ndash 3338},
         url={http://0-dx.doi.org.almirez.ual.es/10.1080/00927872.2011.555850},
      review={\MR{2845575}},
}

\bib{FuGuilHerzogTorrecillas}{article}{
      author={Fu, X.~H.},
      author={Guil~Asensio, P.~A.},
      author={Herzog, I.},
      author={Torrecillas, B.},
       title={Ideal approximation theory},
        date={2013},
        ISSN={0001-8708},
     journal={Adv. Math.},
      volume={244},
       pages={750\ndash 790},
         url={http://0-dx.doi.org.almirez.ual.es/10.1016/j.aim.2013.05.020},
      review={\MR{3077888}},
}

\bib{Herzog}{article}{
      author={Herzog, Ivo},
       title={The phantom cover of a module},
        date={2007},
        ISSN={0001-8708},
     journal={Adv. Math.},
      volume={215},
      number={1},
       pages={220\ndash 249},
         url={http://0-dx.doi.org.almirez.ual.es/10.1016/j.aim.2007.03.010},
      review={\MR{2354989}},
}

\bib{Lam}{book}{
      author={Lam, T.~Y.},
       title={Lectures on modules and rings},
      series={Graduate Texts in Mathematics},
   publisher={Springer-Verlag, New York},
        date={1999},
      volume={189},
        ISBN={0-387-98428-3},
         url={http://0-dx.doi.org.almirez.ual.es/10.1007/978-1-4612-0525-8},
      review={\MR{1653294}},
}

\bib{Stovicek}{article}{
      author={St'ov\'{\i}cek, Jan},
       title={Deconstructibility and the {H}ill lemma in {G}rothendieck
  categories},
        date={2013},
        ISSN={0933-7741},
     journal={Forum Math.},
      volume={25},
      number={1},
       pages={193\ndash 219},
      review={\MR{3010854}},
}

\end{biblist}
\end{bibdiv}
 
\end{document}